\documentclass{amsart}
\usepackage{amsmath, amsfonts, latexsym, array, amssymb, }
\usepackage{psfrag}
\usepackage[all]{xypic}
\usepackage{graphicx}
\usepackage{color}

\newtheorem{thm}{Theorem}[section]
\newtheorem{cor}[thm]{Corollary}
\newtheorem{lem}[thm]{Lemma}

\newtheorem{prop}[thm]{Proposition}

\newtheorem{defn}[thm]{Definition}
\newtheorem{rem}[thm]{Remark}

\usepackage{ulem}

\newcommand{\comment}[1]{}

\begin{document}

\title[Generic Properties Magnetic] {Generic Properties of Magnetic Flows}
\author{Alexander Arbieto and Freddy Castro}
%\address{Department of Mathematics, Brigham Young University, Provo, UT 84602}
\address{Instituto de Matem\'atica, Universidade Federal do Rio de Janeiro, P. O. Box 68530, 21945-970 Rio de Janeiro, Brazil.}
\email{arbieto@im.ufrj.br}
\email{vicente@im.ufrj.br, freddycastrovicente@gmail.com}
\thanks{}

\date{Sep 16, 2015}
\commby{}

\begin{abstract}
We obtain Kupka-Smale's theorem and Franks' lemma for magnetic flows on manifolds with any dimension. This improves Miranda's result \cite{M1,M2} on surfaces. However our methods  relies on geometric control theory, like in Rifford and Ruggiero articles \cite{RR,LRR}.
\end{abstract}

\maketitle

\section{Introduction}

The perturbative theory on dynamical systems is one of the most powerful ones to describe  robust and generic properties of dynamical systems. For instance, the well known Pugh's closing lemma shows that for generic diffeomorphisms the non-wandering set is the closure of periodic points, in the $C^1$-topology. Thus, showing the existence of periodic points for generic diffeomorphisms.

Previous perturbations theorems were Kupka-Smale's theorem \cite{P} and the Franks' lemma \cite{F}. Both deals with the derivative on periodic points (even that, the second part of Kupka-Smale theorem deals with the transversality of the invariant manifolds of periodic points).

Together they can be used to show that stable systems has the property that all of its periodic orbits are hyperbolic. Indeed, any stable diffeomorphism $f$ has a nearby diffeomorphisms which is Kupka-Smale, thus having countable periodic points of the same period (actually finite, since the manifold is compact). However, if there are  diffeomorphisms, nearby to $f$, with a non-hyperbolic periodic orbit, then by Franks' lemma, we are able to obtain another diffeomorphism close to the original, such that locally the dynamics is linear (non hyperbolic). Thus, exhibiting non countable periodic orbits of the same period. Since both dynamics, the Kupka-Smale one and the later one, need to be conjugate, we reach a contradiction.

Those perturbations tools were generalized to more difficult settings: vector fields, Hamiltonians, etc. However,  there are very  important dynamics coming from the differential geometry: geodesic flows. The understanding of this flow quickly became very important. It was Hopf (for surfaces) and Anosov (the general case), that give us tools for the understanding of the dynamics and statistics of that flow for negatively curved Riemannian manifolds.

Since then, perturbative analysis was employed to understand those manifolds in the general case. However, there was a huge obstacle to this understanding. The main one is that for diffeomorphism and vector fields the perturbations are local. However, for metrics (like in the geodesic flow) the perturbations are local in the manifold, but the dynamics lives in the tangent bundle. Thus, the effect of the perturbation spreads in a cylinder on the tangent bundle. This brings more difficulties to control the recurrence of the initial orbits.

The first ones to deal with Kupka-Smale's theorem for geodesic flows were Klingenberg, Takens \cite{KT} and Anosov \cite{An}. Using, Abraham's \cite{Ab} transversality theorems. Both were delicate computations and do not were suitable to obtain a Franks-type result. It was Contreras and Paternain \cite{CP}, the first ones to obtain Franks' lemma for geodesic flows, but for surfaces. Even so they managed to use this to obtain $C^2$-genericity of positive topological entropy for geodesic flows.

After many years, Contreras \cite{C} was able to obtain Franks' lemma for any dimension (geodesic flows), and thus proving the genericity of positive entropy in the general case.

The geodesic equation can be written like 

$$\left\{
\begin{array}{ccc}
\dot{x}&=&v,\\
\nabla_vv&=&0.
\end{array}\right.$$

It can be view like the Hamiltonian flow of the following Hamiltonian in $TM$

$$H(x,v)=\displaystyle\frac{1}{2}\left\langle v,v\right\rangle_x,$$

with the symplectic form

$$\omega_0(\cdot,\cdot)=\left\langle\left\langle  \cdot,\mathcal J\cdot \right\rangle\right\rangle,$$

where $\left\langle \left\langle \cdot,\cdot \right\rangle \right\rangle$ is the \textit{Sasaki metric} and $\mathcal J:TTM\to TTM$ is defined by 

$$\mathcal{J}_\theta=\left( \begin{array}{cc}	0 & I_n \\ -I_n & 0
\end{array}\right).$$

Another well known flows, induced from some physical phenomena, are the magnetic flow and  Gaussian's Thermostats. The first one can be written as the following equation,

$$\left\{
\begin{array}{ccc}
\dot{x}&=&v,\\
\nabla_vv&=&Y_x(v),
\end{array}\right.$$

where $Y$ is the so called Lorentz force. The second one, use a vector field $E$ as a thermostat,

$$\left\{
\begin{array}{ccc}
\dot{x}&=&v,\\
\nabla_vv&=&E-\frac{\left\langle E,v\right\rangle }{\left\langle v,v\right\rangle }v.
\end{array}\right.$$

The perturbation analysis was done for both flows, leading to new interest techniques. The Gaussian thermostat was studied by Latosinski \cite{L} in any dimension, but the magnetic flow was studied by Miranda \cite{M1, M2} but only in dimension two.

Our main results is to give a full proof of Kupka-Smale's theorem and Franks' lemma for magnetic flows in any dimension.

Our method differs to the previous ones, since we use new results from the Control Theory that was used by Rifford and Ruggiero \cite{RR}, to obtain similar results for Hamiltonians perturbations. 

Let us explain the differences. There are three objects in the analysis: the Hamiltonian (that is the kinetic force in the geodesic and magnetic flow), the symplectic form, and the metric. For geodesic flows \cite{C,CP}, the metric is perturbed. In Rifford and Ruggiero article \cite{RR}, the Hamiltonian is perturbed. In our article the symplectic form is perturbed. To obtain these results, one of the key arguments is to obtain nice coordinates to transform our analysis in a equation to apply the control theory.

Let us give more precise statements.

Let $M$ be a closed oriented manifold of dimension $m=n+1$ endowed with a Riemannian metric $g=\left\langle\cdot,\cdot\right\rangle $. Let $\pi:TM\rightarrow M$ denote the canonical projection defined on the tangent bundle. Let $\omega_0$ be the symplectic structure on $TM$ obtained by pulling back the canonical symplectic structure of the cotangent bundle $T^*M$ by the Riemannian metric. Let $H:TM\rightarrow M$ be the Hamiltonian given by

$$H(x,v)=\displaystyle\frac{1}{2}\left\langle v,v\right\rangle_x. \textrm{\ \ \ {\it (Kinetic Energy)}}.$$

Given a smooth closed 2-form $\Omega$ on $M$, we define $\omega=\omega(\Omega)=\omega_0+\pi^*\Omega$, a new symplectic form on $TM$ that is called the \textit{twisted symplectic structure}. The \textit{magnetic field} associated with $\Omega$ is the Hamiltonian field $X=X(\Omega)$ of the Hamiltonian $H$ with respect to $\omega(\Omega)$. The \textit{magnetic flow} associated with $\Omega$ is the Hamiltonian flow $\phi_t=\phi_t(\Omega):TM\rightarrow TM$ induced by the vector field $X$. Observe that this flow preserves the energy levels. This flow models the motion of a unit mass particle under the effect of the \textit{Lorentz force} $Y=Y(\Omega):TM\rightarrow TM$, that is the linear antisymmetric bundle map determined by

$$\Omega_x(u,v)=\left\langle Y_x(u),v\right\rangle _x,$$

for every $u,v\in T_xM$. A curve $t\mapsto (\gamma(t),\dot{\gamma}(t))\subset TM$ is an orbit of $\phi_t$ if and only if the projection $\gamma:\mathbb R\rightarrow M$ satisfies

\begin{eqnarray}\label{mag.geo}
	\dfrac{D}{dt}\dot{\gamma}=Y_{\gamma}(\dot{\gamma}).
\end{eqnarray}

Observe that if $\Omega\equiv 0$, the equation (\ref{mag.geo}) is the geodesic equation for the metric $g$. A curve that satisfies equation (\ref{mag.geo}) is called the \textit{$\Omega$-magnetic geodesic}.

Let $T^cM$ be the bunble given by $T^cM:=H^{-1}(c)$. Suppose that $\theta_t:=\phi_t^\Omega(\theta)$ is a periodic orbit with period $T>0$ in $T^cM$ at the point $\theta$. We say that the $\theta_t$ is \textit{non-degenerate} if the linearized Poincar\'e map $d_\theta P(\Omega)$ has no eigenvalues equal to 1. We say that $\theta_t$ is \textit{hyperbolic} if $d_\theta P(\Omega)$ has no eigenvalues with norm equal to 1 and that $\theta_t$ is \textit{elliptic} if all eigenvalues of $d_\theta P(\Omega)$ have norm one but are not roots of unity. For surfaces, a non-degenerate closed orbit is hyperbolic or elliptic.

Given two hyperbolic periodic orbits $\theta_t$ and $\eta_t$ of the magnetic flow $\phi^\Omega_t$, a \textit{heteroclinic orbits} from $\theta_t$ to $\eta_t$ is an orbit whose $\alpha$-limit is $\theta_t$ and its $\omega$-limit is $\eta_t$. The \textit{strong stable} and \textit{strong unstable manifolds} of the hyperbolic orbit $\theta_t$ at the $\theta\in\theta_t$ are defined as 

$$W^{ss}(\theta)=\{\sigma\in T^cM:\lim_{t\to \infty}d(\theta_t,\phi^\Omega_t(\sigma))=0\},$$

$$W^{us}(\theta)=\{\sigma\in T^cM:\lim_{t\to -\infty}d(\theta_t,\phi^\Omega_t(\sigma))=0\},$$

respectively. The \textit{(weak) stable} and \textit{(weak) unstable manifolds} of the hyperbolic periodic orbit $\theta_t$ are defined as 

$$W^s(\theta_t)=\bigcup_{t\in\mathbb{R}}\phi^\Omega_t(W^{ss}(\theta))\text{ and }W^u(\theta_t)=\bigcup_{t\in\mathbb{R}}\phi^\Omega_t(W^{us}(\theta)),$$

respectively. The sets $W^s(\theta_t)$ and $W^u(\theta_t)$ are $m$-dimensional $\phi^\Omega_t$-invariant immersed submanifolds of $T^cM$ and a heteroclinic orbit is an orbit in the intersection $W^s(\theta_t)\cap W^u(\eta_t)$. If $W^s(\theta_t)$ and $W^u(\eta_t)$ are transversal at $\phi^\Omega_t(\sigma)$, we say that the heteroclinic orbit is \textit{transversal}.

Let $\overline{\Omega}^2(M)$ be the set of all smooth closed 2-form on $M$ endowed with the $C^r$-topology. Recall that a subset $\mathcal{R}\subset\overline{\Omega}^2(M)$ is called a $C^r$-residual if it contains a countable intersection of open and dense subsets in the $C^r$-topology.

\begin{defn}
	We say that a property $P$ is $C^r$-generic for magnetic flows if, for each $c>0$, there exists a subset $\mathcal{R}(c)\subset\overline{\Omega}^2(M)$, such that the following holds.
	
	\begin{enumerate}
		\item The subset $\{\Omega\in\mathcal{R}(c):[\Omega]=[\widetilde{\Omega}]\}$ is $C^r$-residual in $\{\Omega\in\overline{\Omega}^2(M):[\Omega]=[\widetilde{\Omega}]\}$, for all $\widetilde{\Omega}\in\overline{\Omega}^2(M)$.
		\item The flow $\phi_t^\Omega|_{T^cM}$ has the property $P$, for all $\Omega\in\mathcal{R}(c)$.
	\end{enumerate}
\end{defn}

Our first result is the following.

\begin{thm}{(Kupka-Smale)}\label{mainthm}
	The following property:
	\begin{enumerate}
		\item all closed orbits are hyperbolic or elliptic,
		\item all heteroclinic points are transversal
	\end{enumerate}
	are $C^r$-generic for magnetic flows, with $1\leq r\leq\infty$.
\end{thm}

Let $c>0$, $\Omega\in\overline{\Omega}^2(M)$ and $\theta=(x,v)\in T^cM$ with $\gamma:[0,\tau]\to M$ denote a segment of $\Omega$-magnetic geodesic such that $\gamma(0)=x$ and $\dot{\gamma}(0)=v$, where $\tau>0$ small enough. 

Is known that the linearized Poincar\'e map $d_\theta P(\Omega)$ is a symplectic endomorphism of $\mathbb{R}^{2n}\times\mathbb{R}^{2n}$. We can identify
the set of all symplectic endomorphism of $\mathbb{R}^{2n}\times\mathbb{R}^{2n}$ with the symplectic group $Sp(n)$. Using Lemma \ref{coordenadas}, we obtain a coordinate system $(U,\psi)$ of $\gamma$. We define

$$\mathcal{F}=\{d\eta\in\Omega^2(M):supp(d\eta)\subset U\text{ and }d\eta=0\text{ in }\gamma\},$$

and
\begin{eqnarray*}
	S_{\tau,\theta}&:&\mathcal{F} \longrightarrow Sp(n),\\
	& &d\eta \longmapsto d_\theta P(\Omega+d\eta)(\tau).
\end{eqnarray*}

Our next result is the following.

\begin{thm}{(Franks' Lemma)}\label{franks}
Let $c>0$, $\Omega\in\overline{\Omega}^2(M)$ and $\mathcal{U}$ be an open neighborhood of $\Omega$ in the $C^r$ topology $(1\leq r)$. There is $\delta=\delta(c,\Omega,\mathcal{U})>0$ such that, for each $\theta\in T^cM$, $\tau>0$ small enough and $\mathcal{F}$ as defined above, the image of the set $(\mathcal{U}-\Omega)\cap\mathcal{F}$ under the map $S_{\tau,\theta}$ contains a ball of radius $\delta$ centered at $S_{\tau,\theta}(0)$ in $Sp(n)$. Moreover, if $\gamma(t)$ is a closed magnetic geodesic of minimal period $T$, then there is a neighborhood $V\subset M$ of $\gamma([\tau,T])$ such that the image of the set $(\mathcal{U}-\Omega)\cap\{d\eta\in\mathcal{F}:Supp(d\eta-\Omega)\subset U-V\}$ under the map $S_{\tau,\theta}$ contains a ball of radius $\delta$ centered at $S_{\tau,\theta}(0)$ in $Sp(n)$.
\end{thm}

The paper is developed as follows. In the Section 3, we present the special coordinates in a segment de magnetic geodesic and its magnetic matrix curvature. In the Section 4 are presented the results of geometric control theory will be used. Using the previous sections, we obtain the important perturbation theorem, which will be presented and proven in section 5. In the section 6 will prove Franks' Lemma for magnetic flows. Finally, in the sections 7 and 8 prove Kupka-Smale's Theorem for magnetic flows.

\begin{section}{Preliminaries}
We describe in this section our setting. Let us begin by fixing on $M$ a smooth Riemannian metric $g$ with Riemann curvature tensor $R$. Let $\pi:TM\rightarrow M$ denote the canonical projection an let $K:TTM\rightarrow TM$ denote the connection map. The latter is defined given its value on each fiber setting

$$K_\theta(\xi)=\dfrac{DZ}{dt}(0),$$

where $Z:(-\epsilon,\epsilon)\to TM$ verifies $Z(0)=\theta,Z'(0)=\xi$ and $D/dt$ denotes the covariant derivative along $\pi\circ Z$.

It is well know that $TTM$ splits as the direct sum of the vertical and the horizontal subbundles. The vertical fiber on $\theta$ is given by

$$V(\theta)=\ker d_\theta\pi,$$

and the horizontal fiber on $\theta$ is defined by 

$$H(\theta)=\ker K_\theta.$$

Thus $T_\theta TM$ can be identified with $T_{\pi(\theta)}M\oplus T_{\pi(\theta)}M$ and hence we write in the sequel 

$$\xi=(\xi_1,\xi_2),$$

where $\xi_1=\xi_h=d_\theta\pi(\xi)$ and $\xi_2=\xi_v=K_\theta(\xi)$ for every $\xi$ in $T_\theta TM$. The structure symplectic $\omega_0$ described in the introduction can be written as $(\theta=(x,v))$

$$(\omega_0)_{\theta}(\xi,\eta)=\left\langle \xi_1,\eta_2\right\rangle _x-\left\langle \xi_2,\eta_1\right\rangle _x,$$

where $\xi,\eta\in T_\theta TM$. Also can write as

$$(\omega_0)_\theta(\xi,\eta)=\left\langle\left\langle  \xi,\mathcal J_\theta\cdot\eta \right\rangle\right\rangle _\theta,$$
	
where $\left\langle \left\langle \xi,\eta \right\rangle \right\rangle_\theta=\left\langle \xi_1,\eta_1\right\rangle_x+\left\langle \xi_2,\eta_2\right\rangle_x$ called \textit{Sasaki metric} and $\mathcal J:TTM\to TTM$ define by $\mathcal J_\theta(X,Y)=(Y,-X)$.

Fix a closed 2-form $\Omega$ in $M$ define $\omega=\omega(\Omega)=\omega_0+\pi^*\Omega$. Let $Y:TM\to TM$ be the bundle map such that 

$$\Omega_x(u,v)=\left\langle Y_x(u),v\right\rangle_x,$$

for every $u,v\in T_x M$. Remember that $X=X(\Omega)$ denote the symplectic gradient of $H$ with respect to $\omega$. Since the identity

$$d_\theta H(\xi)=\omega_\theta(X(\theta),\xi)=(\omega_0)_\theta(X(\theta),\xi)+\left\langle Y_{\pi(\theta)}\cdot d_\theta\pi(X(\theta)),d_\theta\pi(\xi)\right\rangle _x,$$
	
holds for every $\xi\in T_\theta TM$, the identity $\theta=(x,v)$, where $v\in T_xM$ and 

$$\left\langle \xi_2,v\right\rangle _x=\left\langle X_1(\theta),\xi_2\right\rangle _x-\left\langle X_2(\theta),\xi_1\right\rangle _x+\left\langle Y_x( X_1(\theta)),\xi_1\right\rangle _x,$$

is valid for every $\xi_1,\xi_2\in T_xM$ (obviously we made use of the identification $\xi=(\xi_1,\xi_2)$ as it was explained before and $(X_1,X_2)$ are the horizontal and vertical components of $X$). Therefore

$$X(\theta)=X(x,v)=(v,Y_x(v))\in H(\theta)\oplus V(\theta),$$

for every $\theta=(x,v)\in TM$. It is easily seen from this equation that a curve is an integral curve of $X$ if and only if it is of the form $t\mapsto (\gamma(t),\dot{\gamma}(t))\in TM$ and satisfies the equation 

$$\dfrac{D}{dt}\dot{\gamma}=Y_{\gamma}(\dot{\gamma}),$$ 

which is nothing but Newton's law of motion. Moreover 

\begin{eqnarray}\label{fieldcont}
	\Omega\in\overline{\Omega}^2(M)\mapsto X^\Omega\in\mathcal{X}^r(T^cM),
\end{eqnarray}

is continuous and injective.

Let us derive the Jacobi equation. Denote by $\phi_t=\phi_t^\Omega:TM\to TM$ the flow generated by the symplectic gradient $X$. Take a curve $Z:(-\epsilon,\epsilon)\to TM$ with $Z(0)=\theta, Z'(0)=\xi\in T_\theta TM$ and consider the variation $f(s,t)=\pi(\phi_t(Z(s)))$. Set $J_\xi(t):=\dfrac{\partial f}{\partial s}(0,t)$, $\gamma_s(t):=f(s,t)$ and $\gamma_0=\gamma$.

From the well know identity:

$$\dfrac{D}{ds}\frac{D}{dt}\frac{\partial f}{\partial s}=\dfrac{D}{dt}\frac{D}{dt}\frac{\partial f}{\partial s}+R\left(\dfrac{\partial f}{\partial t},\dfrac{\partial f}{\partial s}\right) \dfrac{\partial f}{\partial t},$$

and

$$\dfrac{D}{dt}\dot{\gamma}_s=Y_{\gamma_s}(\dot{\gamma}_s),$$

denote $J'$ by $\frac{D}{dt}J$, we obtain

$$J'_\xi+R(\dot{\gamma},J_\xi)\dot{\gamma}=\dfrac{D}{ds}\left( Y_{\gamma_s}(\dot{\gamma}_s)\right) .$$

Note that the map $(x,v)\mapsto Y_x(v)$ is a (1,1)-tensor. Thus using the covariant derivative $\nabla$ on (1,1)-tensor induced by the Riemannian connection we obtain

$$\frac{D}{ds}Y(\dot{\gamma}_s)=(\nabla_{J_\xi}Y)(\dot{\gamma}_s)+Y(J'_\xi),$$

and we deduce the Jacobi equation

$$J''_\xi+R(\dot{\gamma},J_\xi)\dot{\gamma}-(\nabla_{J_\xi}Y)(\dot{\gamma})-Y(J'_\xi)=0.$$

Computing the horizontal and vertical components of the differential of the magnetic field, we obtain

$$d_\theta\phi_t^\Omega(\xi)=\left(J_\xi(t),J'_\xi(t)\right),$$

for every $\xi\in T_\theta T^cM$. In particular, of (\ref{fieldcont}), the derivative of the magnetic flow $d_\theta\phi_t^\Omega$ depends continuously on $\Omega\in\overline{\Omega}^2(M)$.

We say that $J$ is a \textit{Jacobi field} under $\Omega$ along $\gamma$ if hold

\begin{eqnarray}\label{jacobi}
\frac{D^2}{d^2t}J(t)+R(\dot{\gamma}(t),J(t))\dot{\gamma}(t)-(\nabla_JY)(\dot{\gamma}(t))-Y\left( \frac{D}{dt}J(t)\right) =0.
\end{eqnarray}

We recall two important equations satisfied by the Lorentz force $Y$. Since for each $x\in M$, the map $Y_x:T_xM\to T_xM$ is antisymmetric with respect to the Riemannian metric, we have that for every $u,v,w\in T_xM$

$$\left\langle(\nabla Y)(u,v),w\right\rangle _x+\left\langle u,(\nabla Y)(w,v)\right\rangle _x=0.$$

Also since $\Omega$ is a closed form one easily checks that for every $u,v,w\in T_xM$ we have

$$\left\langle (\nabla Y)(u,v),w\right\rangle _x+\left\langle (\nabla Y)(v,w),u\right\rangle _x+\left\langle (\nabla Y)(w,u),v\right\rangle _x=0.$$

Then we have that

\begin{eqnarray*}
	\frac{d}{dt}\left\langle J',\dot{\gamma}\right\rangle &=&\left\langle J'',\dot{\gamma}\right\rangle +\left\langle J',Y(\dot{\gamma})\right\rangle \\
	&=&\left\langle -R(\dot{\gamma},J)\dot{\gamma}+(\nabla_JY)(\dot{\gamma})+Y(J'),\dot{\gamma}\right\rangle -\left\langle Y(J'),\dot{\gamma})\right\rangle \\
	&=&\left\langle (\nabla_JY)(\dot{\gamma}),\dot{\gamma}\right\rangle \\
	&=&0.
\end{eqnarray*}

Thus $\left\langle J',\dot{\gamma}\right\rangle $ it is constant, we consider always zero.

Note that if $c>0$ then the vector field $X$ has no singularities in $T^cM$. To simplify the notation, we still denote by $\phi_t$ the restriction of the magnetic flow to the energy level $T^cM.$

We denote by $i(M,g)$ injectivity radius of $(M,g)$ and  for $\Omega$ a smooth closed 2-form in $M$, since for $x\in M$, $\Omega_x:T_xM\times T_xM\to \mathbb{R}$,  

$$\|\Omega_x\|:=\sup\{|\Omega_x(u,v)|:u,v\in T_xM\text{ with }\|u\|=\|v\|=1\},$$

then $|\Omega_x(u,v)|\leq \|\Omega_x\|\|u\|\|v\|$ and

$$\|\Omega\|_{C^0}:=\sup_{x\in M}\|\Omega_x\|.$$

\begin{lem}\label{rad.inj}
	Given $c>0$ and $\Omega\in\overline{\Omega}^2(M)$, let $K=K(c,\Omega)\in\mathbb{R}$ be defined as $K=\min\{1/(\|\Omega\|_{C^0}+1)^2,i(M,g)/2c\}$. Then $\pi\circ\phi^\Omega_t(\theta):[0,K)\to M$ is injective, for every $\theta\in T^cM$.
\end{lem}

The proof of this Lemma is equal to the Lemma 2.1 of Miranda \cite{M1}. The $K(c,\Omega)$  will be called the magnetic injectivity radius.

\end{section}

\begin{section}{Special coordinates and Magnetic curvature}
	
In this section we define the coordinates special type of Fermi coordinates, we obtain a coordinate system of a piece of magnetic geodesic where we present the magnetic curvature matrix. The main reference here is Gouda \cite{G}.

Let us consider $c>0$, $\Omega\in\overline{\Omega}^2(M)$, $\theta=(x,v)$ with $H(\theta)=c$ and $\gamma$ a $\Omega$-magnetic geodesic such that $\gamma(0)=x$ and $\dot{\gamma}(0)=v$. 
	
Let $\theta_t=\phi^\Omega_t(\theta)=(\gamma(t),\dot{\gamma}(t))$ be a periodic orbit of period $T$ in $T^cM$. Let $\Sigma\subset T^cM$ be a local transversal section in the energy level $T^cM$ at the point $\theta$. We say that $\theta_t$ is \textit{non-degenerate} if the linearized Poincar\'e map $d_\theta P(\Omega):T_\theta \Sigma\to T_\theta \Sigma$ has no eigenvalue equalto 1. The linearized Poincar\'e map is a linear symplectic mapping. Let $\delta\Omega\in\overline{\Omega}^2(M)$ such that $(\delta\Omega)_{\gamma(t)}=0$ for every $t\in[0,T]$, then $\phi^{\Omega+\delta\Omega}_t$ preserves the closed orbit $\theta_t$ and its energy level. If $\delta\Omega$ is small enough in a neighborhood of $\gamma([0,T])$, the Poincar\'e return map $ P(\Omega+\delta\Omega):\Sigma\to \Sigma$ associated to the magnetic flow of $\Omega+\delta\Omega$ in $T^cM$ and its differential $d_\theta P(\Omega+\delta\Omega):T_\theta\Sigma\to T_\theta\Sigma$ are well-defined. Our aim is to show that the set of $d_\theta P(\Omega+\delta\Omega)$ for $\delta\Omega$ as above small enough contains as open subset of the set of linear symplectic matrices from $T_\theta\Sigma$ onto itself.	
	
Let $v_1:=v/\sqrt{2c}$ and let us choose $v_2,\ldots,v_m\in T_x M$ so that $v_1,v_2,\ldots,v_m$ is an orthonormal basis in $T_x M$. We define a vector field $V_i$ along $\gamma$ as a solution of the differential equation

$$\left\{
\begin{array}{lc}
	\displaystyle\frac{D}{dt}V_i=Y(V_i),\\
	\\
	V_i(0)=v_i.
\end{array}\right.$$

In particular $V_1=\dot{\gamma}/\sqrt{2c}$. Note that

\begin{eqnarray*}
	\frac{d}{dt}\left\langle V_i,V_j\right\rangle &=&\left\langle V'_i,V_j\right\rangle +\left\langle V_i,V'_j\right\rangle  \\
	&=&\left\langle Y(V_i),V_j\right\rangle +\left\langle V_i,Y(V_j)\right\rangle \\
	&=&\left\langle Y(V_i),V_j\right\rangle -\left\langle Y(V_i), V_j\right\rangle \\
	&=&0.
\end{eqnarray*}

Thus $V_1,\cdots,V_m$ are orthonormal vector fields along $\gamma$ (type Fermi coordinates). %In particular $\left\| \dot{\gamma}\right\| \equiv r^2$.

We know that $Y_x:T_xM\to T_xM$ is an antisymmetric linear mapping for each $x\in M$. Define $Pr_\theta:T_xM\to v^\bot$ the map natural projection, where $v^\bot=\{u\in T_xM:\left\langle u,v\right\rangle _x=0\}$, also define $(Y_\bot)_x:T_xM\to T_xM$ as $(Y_\bot)_x=Pr_\theta Y_x Pr_\theta$, is clear that $(Y_\bot)_x(v)=0$, $(Y_\bot)_x(v^\bot)\subset v^\bot$ and $(Y_\bot)_x^*=-(Y_\bot)_x$ since that $(Pr_\theta)^*=Pr_\theta$. Let $0<\tau<K(c,\Omega)$, for each $t\in[0,\tau]$, define $P_t:T_{\gamma(t)}M\to T_{\gamma(t)}M$ as 

$$P_t=\exp\left(\frac{1}{2}\int_0^t(Y_\bot)_{\gamma(s)}ds\right),$$

it is clear that this map is a linear isomorphism and $P_t^{-1}=P_t^*$ this is an orthogonal linear map then $e_1(t):=P_t^{-1}V_1(t),\ldots,e_m(t):=P_t^{-1}V_m(t)$ is an orthonormal basis of $T_{\gamma(t)}M$.

Consider the differentiable map $\Phi:[0,\tau
]\times \mathbb{R}^n\to M$ given by

$$\Phi(x_1,x_2,\ldots,x_m)=\exp_{\gamma(x_1)}\left[\sum_{i=1}^{m}x_ie_i(x_1)\right],$$

where $\exp_x:T_xM\to M$ denotes the Riemannian exponential map. This map has maximal rank at $(x_1,0,\ldots,0)$, $x_1\in[0,\tau]$. Since $\gamma(t)$ has no self-intersections on $t\in[0,\tau]$, there exists a neighborhood $V$ of $[a,b]\times\{0\}$. Then $\psi^{-1}:=\Phi|_V$ is a diffeomorphism, if $U:=\Phi(V)$ then $(U,\psi)$ is a local coordinate chart where $\gamma(t)=(t,0)$, $g_{ij}(t,0)=\delta_{ij}$ and the Christoffel symbols are $\Gamma^k_{ij}(t,0)=0$. Let 

$$\overline{Y}_{ij}(t):=\left\langle Y_{\gamma(t)}(V_i(t)),V_j(t)\right\rangle_{\gamma(t)} \text{ and }Y_{ij}(t):=\left\langle Y_{\gamma(t)}(e_i(t)),e_j(t)\right\rangle_{\gamma((t))} ,$$ 

denote $\overline{Y}(t)=(\overline{Y}_{ij}(t))$ and $Y(t)=(Y_{ij}(t))$ are the matrices representations of $Y_{\gamma(t)}$ at coordinates $V_i(t)$ and $e_i(t)$ respectively. Thus we have that 

$$Y(t)=P_t^{-1}\overline{Y}(t)P_t.$$ 

In these coordinates note that $e_1(t)=V_1(t)$ 
since $(\overline{Y}_\bot)_{\gamma(t)}$ has zeros in the first column and first row. Moreover note that $P'_t=\displaystyle\frac{1}{2}P_t(Y_\bot)_{\gamma(t)}$.

Let $J$ be a Jacobi field under $\Omega$ along $\gamma$ such that $\displaystyle\frac{D}{dt}J$ is orthogonal to $\dot{\gamma}$. Let $J$ expressed as $J=\displaystyle\sum_{j=1}^mf_je_j$ where each $f_j$ is a smooth function along $\gamma$. Then

\begin{eqnarray*}
	J'&=&\sum_{j-1}^m\left(f'_je_j+f_je'_j\right)\\
	J''&=&\sum_{j-1}^m\left( f''_je_j+2f'_je'_j+f_je''_j \right),
\end{eqnarray*}

but $e'_1=V'_1=\overline{Y}(V_1)=Y(e_1)$ and for $j=2,\ldots m$ we have that

\begin{eqnarray*}
	e'_j&=&(P^{-1})'V_j+P^{-1}V'_j\\
	&=&P^{-1}\overline{Y}(V_j)-\frac{1}{2}P^{-1}\overline{Y}_\bot V_j\\
	&=&P^{-1}\overline{Y}P(e_j)-\frac{1}{2}P^{-1}\overline{Y}_\bot P(e_j)\\
	&=&Y(e_j)-\frac{1}{2}Y_\bot(e_j),
\end{eqnarray*}

observe that, since $Y_\bot(e_1)=0$ then $e'_j=Y(e_j)-\frac{1}{2}Y_\bot(e_j)$ for all $j=1,2,\ldots,m$. Also have that for all $j=1,2,\ldots,m$

\begin{eqnarray*}
	e''_j&=&\nabla_{\gamma'}(Y(e_j))-\frac{1}{2}(Y_\bot(e_j))'\\
	&=&(\nabla_{\gamma'}Y)(e_j)+Y(e'_j)-\frac{1}{2}Y'_\bot(e_j)-\frac{1}{2}Y_\bot(e'_j)\\
	&=&(\nabla_{\gamma'}Y)(e_j)+Y(e'_j)-\frac{1}{2}Y'_\bot(e_j)-\frac{1}{2}Y_\bot Y(e_j)+\frac{1}{4}Y^2_\bot(e_j).
\end{eqnarray*}

Since $J$ is a Jacobi field, this satisfies a equation (\ref{jacobi}) then

\begin{eqnarray*}
	J''+\sum^m_{j=1}\left(f_jR(\gamma',e_j)\gamma'-f_j(\nabla_{e_j}Y)(\gamma')-f'_jY(e_j)-f_jY(e'_j)\right)&=&0,\\ 
	\sum^m_{j=1}\left\lbrace f''_je_j+f'_j(Y-Y_\bot)(e_j)+f_j[R(\gamma',e_j)\gamma'+\right. \\
	\left. \left. (\nabla_{\gamma'}Y)(e_j)-(\nabla_{e_j}Y)(\gamma')-\frac{1}{2}Y'_\bot(e_j)-\frac{1}{2}Y_\bot Y(e_j)+\frac{1}{4}Y^2_\bot(e_j)\right] \right\rbrace &=&0,
\end{eqnarray*}

denote by

\begin{eqnarray*}
	R_{ij}&:=&\left\langle R(\gamma',e_i)\gamma',e_j\right\rangle =\left\langle R(\gamma',e_j)\gamma',e_i\right\rangle =\left\langle e_i,R(\gamma',e_j)\gamma'\right\rangle ,\\
	(Y')_{ij}&:=&\left\langle Y'(e_i),e_j\right\rangle =\left\langle (\nabla_{\gamma'}Y)(e_i),e_j\right\rangle =-\left\langle e_i,(\nabla_{\gamma'}Y)e_j\right\rangle ,\\
	(\partial Y)_{ij}&:=&\sqrt{2c}(\nabla_{e_j}Y)_{i1}=\sqrt{2c}\left\langle (\nabla_{e_j}Y)(e_i),e_1\right\rangle =-\left\langle e_i,(\nabla_{e_j}Y)(\gamma')\right\rangle ,\\
	(Y'_\bot)_{ij}&:=&\left\langle Y'_\bot(e_i),e_j\right\rangle =-\left\langle e_i,Y'_\bot(e_j)\right\rangle ,\\
	(YY_\bot)_{ij}&:=&\left\langle YY_\bot(e_i),e_j\right\rangle =\left\langle e_i,Y_\bot Y(e_j)\right\rangle ,\\
	(Y^2_\bot)_{ij}&:=&\left\langle Y^2_\bot(e_i),e_j\right\rangle =\left\langle e_i,Y^2_\bot(e_j)\right\rangle .
\end{eqnarray*}

Note que $(Y_\bot Y)_{ij}=(YY_\bot)_{ij}$ for all $i,j=2,\ldots,m$, moreover $(YY_\bot)_{1j}=0$ for all $j=1,2,\ldots,m$ and $(YY_\bot)^*=Y_\bot Y$. Thus we have that, if $f=(f_1,f_2,\ldots,f_m)$, then

\begin{eqnarray*} 
	f''+(Y_\bot-Y)f'+(R+\partial Y-Y'+\frac{1}{2}Y'_\bot+\frac{1}{2}YY_\bot+\frac{1}{4}Y^2_\bot)f&=&0.
\end{eqnarray*}

The first line of the above equation is written as

\begin{eqnarray*}
	f''_1+\sum_{j=1}^m(-Y_{1j}f'_j-Y_{1j}f'_j)&=&0,\\
	\frac{d}{dt}(f'_1-\sum_{j=2}^mY_{1j}f_j)&=&0.
\end{eqnarray*}

Since $Y_{11}=0$ and $\left\langle J',\gamma'\right\rangle =0$ then $f'_1=\displaystyle\sum_{j=2}^mY_{1j}f_j$. For $i\neq 1$ we have that

\begin{eqnarray*}
	f''_i-Y_{i1}f'_1+\sum_{j=2}^{m}\left(R_{ij}+(\partial Y)_{ij}-Y'_{ij}+\dfrac{1}{2}(Y'_\bot)_{ij}+\right. \\
	\left. \frac{1}{2}(YY_\bot)_{ij}+\dfrac{1}{4}(Y_\bot^2)_{ij}\right)f_j&=&0,\\
	f''_i+\sum_{j=2}^{m}(R_{ij}+(\partial Y)_{ij}-\frac{1}{2}Y'_{ij}+\frac{3}{4}(Y^2)_{ij}-Y_{i1}Y_{1j})f_j&=&0, 
\end{eqnarray*}

since $Y_{ij}=(Y_\bot)_{ij}$ for all $i,j=2,3,\ldots,m$, and if we denote $\widetilde{Y}_{ij}:=Y_{i1}Y_{1j}$, then the new equation is

\begin{eqnarray}\label{eqf}
	f''+\left( R+\partial Y-\frac{1}{2}Y'+\frac{3}{4}Y^2-\widetilde{Y}\right) f&=&0.
\end{eqnarray}

Note here the matrices are of order $n\times n$, also that $R$, $Y^2$ and $\widetilde{Y}$ are symmetric matrices, also see that $\partial Y-\frac{1}{2}Y'$ is a symmetric matrix, for $i,j=2,3,\ldots,m$ we have that 

$$(\partial Y-\frac{1}{2}Y'_{ij})=\left\langle (\nabla_{e_j}Y)(e_i),\gamma'\right\rangle -\frac{1}{2}\left\langle Y'(e_i),e_j\right\rangle ,$$

then

\begin{eqnarray*}
	(\partial Y-\frac{1}{2}Y')_{ij}-(\partial Y-\frac{1}{2}Y')_{ji}
	=\left\langle (\nabla Y)(e_i,e_j),\gamma'\right\rangle +\\
	\left\langle (\nabla Y)(\gamma',e_i),e_j\right\rangle +\left\langle (\nabla Y)(e_j,\gamma'),e_i\right\rangle=0,
\end{eqnarray*}

since $\Omega$ is closed $(d\Omega=0)$. On the other hand, as $Y_{ij}=\left\langle Y(e_i),e_j\right\rangle =\Omega(e_i,e_j)$ and $Y=\Omega$ seen as matrix. So we define the matrix \textit{magnetic curvature of} $\Omega$ as

\begin{eqnarray}
	K^\Omega(t):=R_{\gamma(t)}+\partial\Omega_{\gamma(t)}-\frac{1}{2}\Omega'_{\gamma(t)}+\frac{3}{4}\Omega^2_{\gamma(t)}-\widetilde{\Omega}_{\gamma(t)},
\end{eqnarray}

it is a symmetric ($n\times n$)-matrix, then

$$f''+K^\Omega f=0.$$

We shall study the real $(n\times n)$-matrix differential equation along $\gamma$,

\begin{eqnarray}
	X''+K^\Omega X=0.
\end{eqnarray}

It is equivalent to 

$$
\left( \begin{array}{c}
X\\X'
\end{array}\right)'=
\left( \begin{array}{cc}
0 & I \\
-K^\Omega & 0
\end{array}\right)
\left(  \begin{array}{c}
X\\X'
\end{array}\right)  
$$

Let $W=\left( \begin{array}{c}
X\\X'
\end{array}\right)$, then

\begin{eqnarray}\label{jac.mat.}
	W'(t)=\left(\begin{array}{cc}
	0 & I \\
	-K^\Omega(t) & 0
	\end{array} \right)W(t).
\end{eqnarray}

Thus, finally we have that 

\begin{lem}\label{coordenadas}
	Let $\theta([0,\tau])$ be a nonsingular orbit of the magnetic flow of $\Omega$ without self-intersection. Exists a local coordinate chart $(U,\psi)$, such that $\psi=(x_1=t,x_2,\ldots,x_m)$, $\psi(x)=\psi(\pi(\theta))=0$ and $\gamma(t)=(t,0,\ldots,0)$, satisfying (\ref{jac.mat.}), where the matrix $W(t)$ represents a basis of Jacobi fields and its derivatives defined in the orbit, and the matrix $K^\Omega(t)$ represents the magnetic curvature.
\end{lem}

In the case of the geodesic flow, we have the same matrix with $K^\Omega$ being the Riemannian curvature matrix which is always a symmetric matrix. In our case does not run Fermi coordinates so we had to make a rotation of the coordinates of Fermi function of $\Omega$ and simultaneously obtain a symmetric matrix. This work is in Gouda \cite{G}

\begin{rem}
	For $m=3$ and $2c=1$, we have that
	
	$$\begin{array}{cc}
	
	\Omega:=\left( \begin{array}{ccc}
	0 & \alpha & -\beta\\
	-\alpha & 0 & \sigma\\
	\beta & -\sigma & 0
	\end{array}\right),
	
	& 
	
	R:=\left( \begin{array}{ccc}
	0 & 0 & 0\\
	0 & a & b\\
	0 & b & c
	\end{array}\right) 
	
	\end{array}$$
	
	then
	
	$$\partial\Omega=\left( \begin{array}{ccc}
	0 & 0 & 0\\
	-\partial_1\alpha & -\partial_2\alpha & -\partial_3\alpha\\
	\partial_1\beta & \partial_2\beta & \partial_3\beta
	\end{array}\right)$$
	
	and $\widetilde{\Omega}=\left(\begin{array}{cc}
	-\alpha^2 & \alpha\beta \\
	\alpha\beta & -\beta^2
	\end{array}\right) $. Thus the equation (\ref{eqf}) is written as
	
	\begin{eqnarray*}
	f''+
	\left(\left( \begin{array}{cc}
	a & b \\
	b & c
	\end{array}\right)
	+
	\left( \begin{array}{cc}
	-\partial_2\alpha & \partial_2\beta+\frac{1}{2}\partial_1\sigma \\
	\partial_2\beta+\frac{1}{2}\partial_1\sigma & \partial_3\beta
	\end{array}\right)
	+\right.\\
	\left.\left(\begin{array}{cc}
	-\dfrac{3}{4}\sigma^2+\alpha^2 & \alpha\beta \\
	\alpha\beta & -\dfrac{3}{4}\sigma^2+\beta^2
	\end{array}\right)
	\right)f&=&0
	\end{eqnarray*}
	
	wiht $f=\left(\begin{array}{c}f^2\\
	f^3\end{array}\right).$
\end{rem}
\end{section}

\begin{section}{Geometric control theory}\label{GCT}

Our aim here is to provide sufficient conditions for first and second order local controllability results. This kind of results could be developed for nonlinear control systems on smooth manifolds. For sake of simplicity, we restrict our attention here to the case of affine control systems on the set of (symplectic) matrices.\\

\textbf{The End-Point mapping.} Let us a consider a \textit{bilinear control system} on $M_{2n}(\mathbb{R})$ (with $n,k\geq 1$), of the form

\begin{eqnarray}\label{control}
	X'(t)=A(t)X(t)+\sum_{i=1}^{k}u_i(t)B_iX(t)\text{ for a.e. }t,
\end{eqnarray}

where the \textit{state} $X(t)$ belongs $M_{2n}(\mathbb{R})$, the \textit{control} $u(t)$ belongs to $\mathbb{R}^k$ and $t\in[0,T]\mapsto A(t)\in M_{2n}(\mathbb{R})$ (with $T>0$) is a smooth maps, and $B_1,\ldots,B_k$ are $k$ matrices in $M_{2n}(\mathbb{R})$. Given $\overline{X}\in M_{2n}(\mathbb{R})$ and $\overline{u}\in L^2([0,T];\mathbb{R}^k)$, the Cauchy problem 

$$\left\lbrace \begin{array}{l}
	X'(t)=A(t)X(t)+\displaystyle\sum_{i=1}^{k}\overline{u}_i(t)B_iX(t)\text{ for a.e. }t\in[0,T],\\
	X(0)=\overline{X},
\end{array}	\right. $$

possesses a unique solution $X_{\overline{X},\overline{u}}(\cdot)$. The \textit{End-Point mapping} associated with $\overline{X}$ in time $T>0$ is defined as 

$$\begin{array}{ccccc}
	E^{\overline{X},T}&:&L^2([0,T];\mathbb{R}^k)&\longrightarrow & M_{2n}(\mathbb{R})\\
	&&u&\longmapsto&X_{\overline{X},u}(T).
\end{array}$$

It is a smooth mapping whose differential can be expressed in terms of the linearized control systems. Given $\overline{X}\in M_{2n}(\mathbb{R})$, $\overline{u}\in L^2([0,T];\mathbb{R}^k)$, and setting $\overline{X}(\cdot):=X_{\overline{X},\overline{u}}(\cdot)$, the differential of $E^{\overline{X},T}$ at $\overline{u}$ is given by the linear operator

$$\begin{array}{ccccc}D_{\overline{u}}E^{\overline{X},T}&:&L^2([0,T];\mathbb{R}^k)&\longrightarrow & M_{2n}(\mathbb{R})\\&&v&\longmapsto&Y(T),
\end{array}$$

where $Y(\cdot)$ is the unique solution to the Cauchy problem

$$\left\lbrace \begin{array}{l}
Y'(t)=A(t)Y(t)+\displaystyle\sum_{i=1}^{k}v_i(t)B_i\overline{X}(t),\text{ for a.e. } t\in [0,T],\\ 
Y(0)=0.
\end{array}	\right. $$

Note that if we denote by $S(\cdot)$ the solution to the Cauchy problem

$$\left\lbrace \begin{array}{ccl} S'(t) & = & A(t)S(t)\text{ for every }t\in[0,T],\\ S(0) & = &I_{2n},
\end{array}	\right. $$	

 the there holds

$$D_{\overline{u}}E^{\overline{X},T}\cdot v=\sum_{i=1}^{k}S(T)\int_{0}^{T}v_i(t)S(t)^{-1}B_i\overline{X}(t)dt,$$

for every $v\in L^2([0,T];\mathbb{R}^k)$.

Let $Sp(n)$ be the symplectic group in $M_{2n}(\mathbb{R})$ $(n\geq 1)$, that is the smooth submanifold of matrices $X\in M_{2n}(\mathbb{R})$ satisfying

$$X^*\mathbb{J} X=\mathbb{J}\text{ where }\mathbb{J}=\left( \begin{array}{cc}	0 & I_n \\ -I_n & 0
\end{array}\right).$$

$Sp(n)$ has dimension $p:=n(2n+1)$. Denote by $S(2n)$ the set of symmetric matrices in $M_{2n}(\mathbb{R})$. The tangent spaces to $Sp(n)$ at the identity matrix is given by 

$$T_{I_{2n}}Sp(n)=\{Y\in M_{2n}(\mathbb{R}):\mathbb{J}Y\in S(2n)\}.$$

Therefore, if there holds

\begin{eqnarray}\label{condition}
	\mathbb{J}A(t),\mathbb{J}B_1,\ldots,\mathbb{J}B_k\in S(2n) \text{ for all }t\in[0,T],
\end{eqnarray}

then $Sp(n)$ is invariant with respect to (\ref{control}), that is for every $\overline{X}\in Sp(n)$ and $\overline{u}\in L^2([0,T];\mathbb{R}^k)$,

$$X_{\overline{X},\overline{u}}(t)\in Sp(n)\text{ for all }t\in[0,T].$$

In particular, this means that for every $\overline{X}\in Sp(n)$, the End-Point mapping $E^{\overline{X},T}$ is valued in $Sp(n)$. Given $\overline{X}\in Sp(n)$ and $\overline{u}\in L^2([0,T];\mathbb{R}^k)$, we are interested in local controllability propeties of (\ref{control}) around $\overline{u}$. The control systems (\ref{control}) is called \textit{controllable around} $\overline{u}$ in $Sp(n)$ (in time $T$) if for every final state $X\in Sp(n)$ close to $X_{\overline{X},\overline{u}}(T)$ there is a control $u\in L^2([0,T];\mathbb{R}^k)$ which steers $\overline{X}$ to $X$, that is such that $E^{\overline{X},T}(u)=X$. Such a property is satisfied as soon as $E^{\overline{X},T}$ is locally open at $\overline{u}$.\\

\textbf{First order controllability results.} Given $T>0$, $\overline{X}\in Sp(n)$ a mapping $t\in[0,T]\mapsto A(t)\in M_{2n}(\mathbb{R})$, $k$ matrices $B_1,\ldots,B_k\in M_{2n}(\mathbb{R})$ satisfying (\ref{condition}), and $\overline{u}\in L^2([0,T];\mathbb{R}^k)$ we say that the control systems (\ref{control}) is \textit{controllable at first order around} $\overline{u}$ in $Sp(n)$ if the mapping $E^{\overline{X},T}:L^2([0,T],\mathbb{R}^k)\to Sp(n)$ is a \textit{submersion} at $\overline{u}$, that is if the linear operator

$$D_{\overline{u}}E^{\overline{X},T}:L^2([0,T];\mathbb{R}^k)\to T_{\overline{X}(T)}Sp(n),$$

is surjective (with $\overline{X}(T):=X_{\overline{X},\overline{u}}(T)$). The following sufficient condition for first order controllability is given in [Rifford-Ruggiero]

\begin{prop}\label{propcontrol}
	Let $T>0,t\in[0,T]\mapsto A(t)$ a smooth mapping and $B_1,\ldots,B_k\in M_{2n}(\mathbb{R})$ satisfying (\ref{condition}). Define the $k$ sequences of smooth mappings 
	$$\{B^j_1\},\ldots,\{B^j_k\}:[0,T]\longrightarrow T_{I_{2n}}Sp(n)$$
	by
	\begin{eqnarray}\label{seq.}
	\left\lbrace 
		\begin{array}{l}
		B^0_i(t)=B_i(t)\\
		B^j_i(t)=\dot{B}^{j-1}_i(t)+B^{j-1}_i(t)A(t)-A(t)B^{j-1}_i(t),
	\end{array}
	\right.
	\end{eqnarray}
	for every $t\in [0,T]$ and every $i=1,\ldots,k$. Assume that there exists some $\overline{t}\in[0,T]$ such that 
	\begin{eqnarray}
		Span\left\lbrace B^j_i(\overline{t}):i=1,\ldots,k,j\in\mathbb{N}\right\rbrace=T_{I_{2n}}Sp(n). 
	\end{eqnarray}
	Then for every $\overline{X}\in Sp(n)$, the control system (\ref{control}) is controllable at first order around $\overline{u}\equiv0$.
\end{prop}

The control system which is relevant in the present paper is not always controllable at first order. We need sufficient condition for controllability et second order.

\textbf{Second-order controllability results.} Using the same notations as above we say that the control system (\ref{control}) is \textit{controllable at second order around} $\overline{u}$ in $Sp(n)$ if there are $\mu,K>0$ such that for every $X\in B(\overline{X}(T),\mu)\cap Sp(n)$, there is $u\in L^2([0,T],\mathbb{R}^k)$ satisfying

$$E^{\overline{X},T}(u)=X\text{ and }\|u\|_{L^2}\leq K|X-\overline{X}(T)|^{1/2}.$$

Obtaining such a property requires a study of the End-Point mapping at second order. Recall that given two matrices $B,B'\in M_{2n}(\mathbb{R})$, the bracket $[B,B']$ is the matrix of $M_{2n}(\mathbb{R})$ defined as 

$$[B,B']:=BB'-B'B.$$

The following results are the key points in the proof of our main theorem.

\begin{prop}\label{seg.ord.}
	Let $T>0,t\in[0,T]\to A(t)$ a smooth mapping and $B_1\ldots,B_k\in M_{2n}(\mathbb{R})$ satisfying (\ref{condition}) such that 
	
	$$B_iB_j=0\text{ for every } i,j=i,\ldots,k.$$
	
	Define the $k$ sequences of smooth mapping $\{B^j_1\},\ldots,\{B^j_k\}:[0,T]\to T_{I_{2n}}Sp(n)$ by (\ref{seq.}) and assume that the following properties are satisfied with $\overline{t}=0:$
	
	$$[B_i^j(\overline{t}),B_i]\in Span\{B^s_r(\overline{t}):r=1,\ldots,k,s\geq 0\}$$
	
	for every $i=1,\ldots,k$, $j=1,2$, and
	
	$$Span\{B_i^j(\overline{t}),[B_i^1(\overline{t}),B_l^1(\overline{t})]:i,l=1,\ldots,k\text{ and }j=0,1,2\}=T_{I_{2n}}Sp(n).$$
	
	Then, for every $\overline{X}\in Sp(n)$, the control system (\ref{control}) is controllable at second order around $\overline{u}\equiv 0.$
\end{prop}

We will need the following parametrized version of Proposition \ref{seg.ord.} which will follow from the fact that smooth controls with support in $(0,T)$ are dense in $L^2([0,T],\mathbb{R}^k)$ and compactness.

\begin{prop}\label{parameters}
	Let $T>0$, and for every $\theta$ in some set of parameters $\Theta$ let $t\in[0,T]\to A^\theta(t)$ be a smooth mapping and $B^\theta_1,\ldots,B^\theta_k\in M_{2n}(\mathbb{R})$ satisfying (\ref{condition}) such that 
	
	\begin{eqnarray}\label{cond1}			B^\theta_iB^\theta_j=0 \text{ for every }i,j=1,\ldots,k.
	\end{eqnarray}
	
	Define for every $\theta\in\Theta$ the $k$ sequences of smooth mapping $\{B_1^{\theta,j}\},\ldots,\{B_k^{\theta,j}\}:[0,T]\to T_{I_{2n}}Sp(n)$ as in (\ref{seq.}) and assume that the following properties are satisfied with $\overline{t}=0$ for every $\theta\in\Theta$:
	
	\begin{eqnarray}\label{cond2} 
		[B^{\theta,j}_i(\overline{t}),B^\theta_i]\in Span\{B^{\theta,s}_r(\overline{t}):r=1,\ldots,k,s\geq 0\},
	\end{eqnarray}
	
	for every $i=1,\ldots,n$, $j=1,2$ and 
	
	\begin{eqnarray}\label{cond3}
	\end{eqnarray}
		$$Span\{B^{\theta,j}_i(\overline{t}),[B^{\theta,1}_i(\overline{t}),B^{\theta,1}_l(\overline{t})]:i,l=1,\ldots,k\text{ and }j=0,1,2\}=T_{I_{2n}}Sp(n).$$

	Assume moreover, that the sets
	
	$$\{B^\theta_i:i=1,\ldots,k,\theta\in\Theta\}\subset M_{2n}(\mathbb{R})$$
	
	and
	
	$$\{t\in[0,T]\mapsto A^\theta(t):\theta\in\Theta\}\subset C^2([0,T];M_{2n}(\mathbb{R}))$$
	
	are compact. Then, there are $\mu,K>0$ such that for every $\theta\in\Theta$, every $\overline{X}\in Sp(n)$ and every $X\in B(\overline{X}^\theta(T),\mu)\cap Sp(n)$ ($\overline{X}^\theta(T)$ denotes the solution at time $T$ of the control system (\ref{control}) with parameter $\theta$ starting from $\overline{X}$), there is $u\in C^\infty([0,T];\mathbb{R}^k)$ with support in $[0,T]$ satisfying
	
	$$E^{\overline{X},T}_\theta(u)=X\text{ and }\|u\|_{C^2}\leq K|X-\overline{X}(T)|^{1/2}$$
	
	($E^{\overline{X},T}_\theta$ denotes the End-Point mapping associated with the control system (\ref{control}) with parameter $\theta$). 

\end{prop}

\end{section}
 
\begin{section}{Local perturbations of the magnetic flow}

In this section we obtain the perturbation result, which is the heart of this work. Our main reference is here Rifford and Ruggiero \cite{RR}.

Applying Lemma \ref{coordenadas} to a piece of a closed orbit $\theta$ for the magnetic flow of $\Omega$, we may assume that $\psi(x)=0$ and $d_x\psi\cdot v=(1,0...,0)$, then  $\theta_t=\phi^\Omega_t(x,v)=(\psi^{-1}(t,0,\ldots0),(d_x\psi)^{-1}(1,0,\ldots,0)):[0,\tau]\to T^cM$, for some $0<\tau<K=K(c,\Omega)$. We need to study generic perturbations of $\Omega$ in the neighborhood $U$ of $\gamma$.

Let $\delta>0$ fix such that $\psi([0,\tau]\times (-\delta,\delta)^n)\subset U$. Let a family of smooth function $u_{ij}:[0,\tau]\to\mathbb{R}$ such that

$$Supp(u_{ij})\subset (0,\tau)\text{ for every }i\leq j\text{ in }2,\ldots,m.$$

We consider $f:[0,+\infty)\to [0,+\infty)$ a smooth function bump such that $f(\lambda)=1$ if $3\lambda\leq1$ and $f(\lambda)=0$ if $3\lambda\geq 2$, we define a family of smooth perturbations $f_i:M\to\mathbb{R}$ with support in $\psi([0,\tau]\times (-\delta,\delta)^n)$ by

\begin{eqnarray*}
	f_1(\psi(x_1,x_2,\ldots,x_m))&=&-\frac{1}{\sqrt{2c}}\sum_{i<j=2}^{m}u_{ij}(x_1)x_ix_jf(\|(x_2,\ldots,x_m)\|),\\	
	f_i(\psi(x_1,x_2,\ldots,x_m))&=&\dfrac{1}{\sqrt{2c}}\int_{0}^{x_1}u_{ii}(s)dsx_if(\|(x_2,\ldots,x_m)\|),
\end{eqnarray*}

for $i=2,\ldots,m$. Now consider the 1-form in $M$ define by $\eta=\displaystyle
\sum^{m}_{k=1}f_kdx_k$ with support in $\psi([0,\tau]\times (-\delta,\delta)^n)$. Then taking $\delta\Omega:=d\eta$ with support in $\psi([0,\tau]\times (-\delta,\delta)^n)$ in coordinates as 

$$(d\eta)_{21}=-\frac{1}{\sqrt{2c}}\sum^{m}_{l=2}u_{2l}(x_1)x_lf(\|(x_2,\ldots,x_m)\|),$$

and for $i=3,\ldots,m$ have that

\begin{eqnarray*}
	(d\eta)_{i1}&=&-\frac{1}{\sqrt{2c}}\left(\sum_{l=2}^{i-1}u_{li}(x_1)x_lf(\|(x_2,\ldots,x_m)\|)+\right. \\
	&&\left.\sum_{l=i}^{m}u_{il}(x_1)x_lf(\|(x_2,\ldots,x_m)\|)\right),
\end{eqnarray*}

and $(d\eta)_{ij}=0$ otherwise. Thus we have that $(x_1=t)$

$$U(t):=\frac{1}{2}(d\eta)'-\partial(d\eta)=\left( \begin{array}{ccccc}
u_{22}(t) & u_{23}(t) & \ldots & u_{2m}(t) \\\\
u_{23}(t) & u_{33}(t) & \ldots & u_{3m}(t) \\\\
\vdots & \vdots & \ddots & \vdots \\\\
u_{2m}(t) & u_{3m}(t) & \ldots & u_{mm}(t)
\end{array}\right) $$\\

a symmetric $n\times n$-matrix. Note that $d\eta_{\gamma(t)}=0$ for all $t\in[0,\tau]$, then

\begin{eqnarray}\label{pert.}
	K^{\Omega+d\eta}(t)=K^\Omega(t)-U(t),
\end{eqnarray}

and the cohomology class $[d\eta]=0$ this is $[\Omega]=[\Omega+\delta\Omega]$ in $H^2(M,\mathbb{R})$. Since $f(\|(x_2,\ldots,m)\|)$ and their derivatives vanish along the segment $\gamma((0,\tau))$, the trajectory $\theta_t$ is an orbit of the magnetic flow of $\Omega+\delta\Omega$ and the level energy is preserved. Using Lemma \ref{coordenadas} in (\ref{jac.mat.}) and by the Jacobi equation, we have that 

$$d_\theta P(\Omega+\delta\Omega)(\tau)(J(0),J'(0))=(J(\tau),J'(\tau)),$$

were $J:[0,\tau]\to \mathbb{R}^n$ is solution to the Jacobi equation

$$J''(t)+K^{\Omega+\delta\Omega}(t)J(t)=0,\text{ for every }t\in[0,\tau].$$

In other terms, $d_\theta P(\Omega+\delta\Omega)(\tau)$ is equal to the $n\times n$ symplectic matrix $X(\tau)$ given by the solution $X:[0,\tau]\to Sp(n)$ at time $\tau$ of the following Cauchy problem:

$$\left\lbrace \begin{array}{l}
X'(t)=A(t)X(t)+\displaystyle\sum_{i\leq j=2}^{m}u_{ij}(t)\mathcal{E}(ij)X(t),\text{ for all }t\in[0,\tau],\\
X(0)=I_{2n},
\end{array}	\right. $$

where the $2n\times 2n$ matrices $A(t),\mathcal{E}(ij)$ are defined by 

$$A(t):=\left( \begin{array}{cc}
	0 & I_n \\
	-K^\Omega(t) & 0
\end{array}\right)\text{ for every }t\in[0,\tau]$$

and 

$$\mathcal{E}(ij):=\left( \begin{array}{cc}
0 & 0 \\
E(ij) & 0
\end{array}\right),$$

where the $E(ij),$ $2\leq i\leq j\leq m$ are the symmetric $n\times n$ matrices defined by 

$$(E(ij))_{k,l}=\delta_{ik}\delta_{jl}+\delta_{il}\delta_{jk}, \text{ for every }i,j=2,\ldots,m.$$

Since our control system has the form (\ref{control}), all the results gathered in Section \ref{GCT} apply. By compactness of $M$ and regularity of the magnetic flow, the compactness assumption in Proposition \ref{parameters} are satisfied. It remains to check that assumptions (\ref{cond1}), (\ref{cond2}) and (\ref{cond3}) hold. 

First we check immediately that 

$$\mathcal{E}(ij)\mathcal{E}(kl)=0,\text{ for every }i,j,k,l\in\{2,\ldots,m\}\text{ with }i\leq j,k\leq l.$$

So, assumption (\ref{cond1}) is satisfied. Since the $\mathcal{E}(ij)$ do not depend on time, we check easily that the matrices $B^0_{ij},B^1_{ij},B^2{ij}$ associated to our system are given by

$$\left\lbrace\begin{array}{c} B^0_{ij}(t)=B_{ij}:=\mathcal{E}(ij) \\ \\
		B^1_{ij}(t)=[\mathcal{E}(ij),A(t)] \\ \\
		B^2_{ij}(t)=[[\mathcal{E}(ij),A(t)],A(t)],
\end{array}\right. $$

for every $t\in[0,\tau]$ and any $i,j=2,\ldots,m$ with $i\leq j$. An easy computation yields for any $i,j=2,\ldots,m$ with $i\leq j$ and any $t\in[0,\tau]$,

$$B^1_{ij}(t)=[\mathcal{E}_{ij},A(t)]=\left( \begin{array}{cc} -E(ij) & 0 \\
		0 & E(ij) \end{array} \right) $$

and

$$B^2_{ij}(t)=[[\mathcal{E}_{ij},A(t)],A(t)]=\left(\begin{array}{cc} 0 & -2E(ij) \\
-E(ij)K^\Omega(t)-K^\Omega(t)E(ij) & 0 \end{array} \right).$$

Then we get for any $i,j=2,\ldots,m$ with $i\leq j$,

$$[B^1_{ij}(0),B_{ij}]=2\left(\begin{array}{cc} 0 & 0 \\
(E(ij))^2 & 0 \end{array}\right)\in Span\left\lbrace B^0_{rs}(0):r\leq s\right\rbrace $$

and 

$$[B^2_{ij}(0),B_{ij}]=2\left(\begin{array}{cc} -(E(ij))^2 & 0 \\
0 & (E(ij))^2 \end{array}\right)\in Span\left\lbrace B^1_{rs}(0):r\leq s\right\rbrace.$$

So assumption (\ref{cond2}) is satisfied. It remains to show that (\ref{cond3}) holds. We first notice that for any $i,j,k,l,=2,\ldots,m$ with $i\leq j,k\leq l$, we have

\begin{eqnarray*}
	[B^1_{ij}(0),B^1_{kl}(0)]&=&[[\mathcal{E}(ij),A(0)],[\mathcal{E}(kj),A(0)]]\\	
	&=&\left(\begin{array}{cc} [E(ij),E(kl)]&0\\0&[E(ij),E(kl)],
	\end{array}\right) 
\end{eqnarray*}

with

\begin{eqnarray}\label{efe}
	[E(ij),E(kl)]=\delta_{il}F(jk)+\delta_{jk}F(il)+\delta_{ik}F(jl)+\delta_{jl}F(ik),
\end{eqnarray}

where $F(pq)$ is the $n\times n$ skew-symmetric matrix defined by

$$(F(pq))_{rs}=\delta_{rp}\delta_{sq}-\delta_{rq}\delta_{sp}.$$

It is sufficient to show that the space $S\subset M_{2n}(\mathbb{R})$ given by

$$S:=Span\left\lbrace B^0_{ij}(0),B^1_{ij}(0),B^2_{ij}(0),[B^1_{kl}(0),B^1_{rs}(0)]:i,j,k,l,r,s\right\rbrace\subset T_{I_{2n}}Sp(n)$$

has dimension $p$. First since the set matrices $\mathcal{E}(ij)$ with $i,j=2,\ldots,m$ with $i\leq j$ forms a basis of the vector space of $n\times n$ symmetric matrices $S(n)$ we check easily by the formulas that the vector space

$$S_1:=Span\{B_{ij},B^2_{ij}(0):i,j\}=Span\left\lbrace \mathcal{E}(ij),[[\mathcal{E}(ij),A(t)],A(t)]:i,j\right\rbrace $$

has dimension $n(n+1)$, We check easily that the vector spaces 

$$S_2:=Span\{B^1_{ij}(0):i,j\}=Span\{[\mathcal{E}(ij),A(0)]:i,j\}$$

and

$$S_3:=Span\{[B^1_{ij}(0),B^1_{kl}(0)]:i,j,k,l\}=Span\{[[\mathcal{E}(ij),A(0)],[\mathcal{E}(kl),A(0)]]:i,j,k,l\}$$

are orthogonal to $S_1$ with respect to the scalar product $P\cdot Q=tr(P^*Q)$. So, we need to show that $S_2+S_3$ has dimension $n^2$. By the above formulas, we have

$$S_2=Span\left\lbrace \left(\begin{array}{cc}-E(ij)&0\\0&E(ij)\end{array}\right):i,j\right\rbrace $$

and

$$S_3=Span\left\lbrace \left(\begin{array}{cc}[E(ij),E(kl)]&0\\0&[E(ij),E(kl)]\end{array}\right):i,j,k,l\right\rbrace,$$

and in addition $S_2$ and $S_3$ are orthogonal. Then first space $S_2$ has the same dimension as $S(n)$, that is $n(n+1)/2$. Moreover, by (\ref{efe}) for every $i\neq j,k=i$ and $l\notin\{i,j\}$, we have 

$$[E(ij),E(kl)]=F(jl).$$

The space spanned by the matrices of the form 

$$\left(\begin{array}{cc}F(jl)&0\\0&F(jl)\end{array} \right),$$

with $2\leq j< l\leq m$ has dimension $n(n-1)/2$. This shows that $S_3$ has dimension at least $n(n-1)/2$ and so $S_2\oplus S_3$ has dimension $n^2$. Thus we have proved the following result.

Let $\mathcal{F}$ the set of $d\eta$ where the $\eta\in\Omega^1(M)$ defined as above, consider

\begin{eqnarray*}
	S_{\tau,\theta}&:&\mathcal{F} \longrightarrow Sp(n),\\
	 & &d\eta \longmapsto d_\theta P(\Omega+d\eta)(\tau).
\end{eqnarray*}

\begin{thm}\label{pert.}
	Let $c>0$ and $\Omega\in\overline{\Omega}^2(M)$ and $0<\tau<K(c,\Omega)$.  There is $\overline{\delta},K>0$ (depending on $c,\Omega$ and $\tau$) such that the
	following property holds:
	
	For each $\theta\in T^cM$, $\mathcal{F}$ as defined above, and $\delta\in(0,\overline{\delta})$,
	
	$$B(S_{\tau,\theta}(0),\delta K)\cap Sp(n)\subset S_{\tau,\theta}\left( B_{C^r}(0,\delta)\cap\mathcal{F}\right). $$
	
	where $B_{C^r}(0,\delta)\subset\overline{\Omega}^2(M)$ is the open ball of radius $\delta$ centred at $0\in\overline{\Omega}^2(M)$ in the $C^r$ topology $(1\leq r)$.

\end{thm}

This is the technical result we need to demonstrate our results.

\end{section}

\begin{section}{Franks' lemma for magnetic flows}
	
In this section we will show how to deduce Theorem \ref{franks} from the technical result, Theorem \ref{pert.} of the previous section.

Let $c>0$, $\Omega\in\overline{\Omega}^2(M)$ and $\mathcal{U}$ be an $C^r$-open neighbourhood of $\Omega$, small enough, in $\overline{\Omega}^2(M)$. We set $\theta=(x,v)\in T^cM$, with $\gamma:[0,\tau]\to M$ magnetic geodesic such that $\gamma(0)=x$ and $\dot{\gamma}(0)=v$, where $0<\tau<K(c,\Omega)$.

Considering the definitions of $\mathcal{F}$ and $S_{\tau,\theta}$ of the previous section, under these conditions we can use the Theorem \ref{pert.}. In this case, there is $r>0$ such that 

$$B(S_{\tau,\theta}(0,r)\cap Sp(n)\subset S_{\tau,\theta}\left( (\mathcal{U}-\Omega)\cap\mathcal{F}\right).$$

This proves the Franks' lemma for magnetic flows. An application of this result is as follows:

Suppose that $\theta_t=(\gamma(t),\dot{\gamma}(t))\subset T^cM$ is a closed orbit and let $T>0$ be its minimal period. By Lemma \ref{rad.inj}, $K:=K(c,\Omega)<T_\theta$ and the number of self-intersection points of $\gamma$ is finite. We fix $\tau\in(K/2,K]$, such that $T_\theta=l\tau$, with $l\in\mathbb{N}$, denote $\gamma_i(t)=\gamma(t+i\tau)$. Then we choose $U_i\subset M$ open and disjoint sets for $0\leq i\leq l-1$, such that 

$$U_i\cap\gamma_i((K/2,\tau))\neq\emptyset,\text{ and } U_i\cap U_j=\emptyset,\text{ for every }i\neq j.$$

For $U=\displaystyle\bigcup_{i=0}^{l-1}U_i$, we consider the map 

$$S_\theta:d\eta\in\mathcal{F}\longmapsto d_\theta P(\Omega+d\eta)(T_\theta)=\displaystyle\prod_{i=0}^{l-1}d_{\theta_{i\tau}}P_i(\Omega+d\eta)\in\displaystyle\prod_{i-0}^{l-1}Sp(n),$$

where $P_i$ is the Poincar\'e map from $\Sigma_{i\tau}$ to $\Sigma_{(i+1)\tau}$. Applying $l$ times Theorem \ref{pert.}, we prove the following corollary.

\begin{cor}\label{pert.corl.}
	Given $\Omega\in\overline{\Omega}^2(M)$ and $\mathcal{U}$ an open neighborhood of $\Omega$ in the $C^r$ topology with $r\geq 1$. Suppose that $\theta_t\subset T^cM$ is a closed orbit with minimal period $T_\theta$. Then choosing $\tau,l$ and $U$ as above, the image of the set $(\mathcal{U}-\Omega)\cap\mathcal{F}$ be the map $S_\theta$ is an open neighborhood of $S_\theta(0)$ in $\displaystyle\prod_{i=0}^{l=1}Sp(n)$.
\end{cor}

This result will be useful in the next section where we start with the proof of Kupka-Smale's Theorem.

\end{section}

\begin{section}{Kupka-Smale theorem for magnetic flows}

In this section we begin with the proof of the Theorem \ref{mainthm}, we will prove the first part here and the second part we leave to the next section. Our main reference is Miranda \cite{M1}, who worked in the same result in surfaces.

Let $\mathcal{N}(t)=\mathcal{N}(\theta_t)\subset T_\theta T^cM$ be the subspace

$$\mathcal{N}(t):=\left\lbrace \xi\in T_{\theta_t}T^cM:\left\langle\xi_1,\dot{\gamma}(t) \right\rangle_{\gamma(t)}=0\right\rbrace .$$

If $\xi=X^\Omega(\theta_t)$, then $\xi_1=\dot{\gamma}(t)$, therefore the subspace $\mathcal{N}(t)$ is transversal to $X^\Omega$ along of $\theta_t$, note that $V(\theta_t)\subset\mathcal{N}(t)$. Hence

$$T_{\theta_t}T^cM=\mathcal{N}(t)\oplus\left\langle X^\Omega(\theta_t) \right\rangle. $$

Therefore, the restriction of the twisted form $\omega_\theta$ to $\mathcal{N}(\theta)$ is a non-degenerate 2-form. Note that $\mathcal{N}(\theta)$ does not depend on the 2-form $\Omega$. For $i=2,\ldots,m$, we have that $(e_i(t),0),(0,e_i(t))\in H(\theta_t)\oplus V(\theta_t)$, then $(e_i(t),0),(0,e_i(t))\in\mathcal{N}(t)$ and

\begin{eqnarray*}
	\omega_{\theta_t}((e_i(t),0),(e_j(t)),0)&=&\Omega_{ij},\\
	\omega_{\theta_t}((0,e_i(t)),(0,e_j(t)))&=&0\text{ and}\\
	\omega_{\theta_t}((e_i(t),0),(0,e_j(t)))&=&\delta_{ij}.
\end{eqnarray*}

Thu, we have that

$$(e_2(t),0),\ldots(e_m(t),0),(0,e_2(t)),\ldots(0,e_m(t)),$$

is an basis of $\mathcal{N}(t).$

We say that a closed orbit is \textit{non-degenerate of ordem $k\in\mathbb{N}$}, if the derivate of the $k$th iterated on the linearized Poincar\'e map has no eigenvalues equal 1. Given $a,c>0$ and $k\in\mathbb{N}$, let $\mathcal{G}^k(c,a)$, be the subset of every $\Omega\in\overline{\Omega}^2(M)$ such that all closed orbits of $\phi_t^\Omega|_{T^cM}$, with minimal period $<a$, are non-degenerate of order $k$. Thus the first part of the Theorem \ref{mainthm} can be reduces to following proposition.

\begin{prop}\label{nondeg.}
	Given $c,a>0$ and $r\in\mathbb{N}$, the subset $\mathcal{G}^1(c,a)\subset\overline{\Omega}^2(M)$ is a open and dense subset in the $C^r$ topology. Moreover, for each $\Omega\in\overline{\Omega}^2(M)$, the subset $\mathcal{G}_{[\Omega]}^1(c,a) $ is $C^r$-dense subset of $\overline{\Omega}_{[\Omega]}^2(M)$.
\end{prop}

Let $\theta_t=(\gamma(t),\dot{\gamma}(t))=\phi^\Omega_t(\theta)$ a closed orbit of minimal period $T_\theta>0$ in $T^cM$, for each $i=2,\ldots,m$ consider a function $u_i\in C^\infty(M)$ with support in a neighborhood of $\gamma([0,T])$ and defined

$$f_i(x_1,\ldots,x_m)=\frac{1}{\sqrt{2c}}\int_{0}^{x_1}u_{i}(s)dsx_i$$

in local coordinates. Let $\eta_i:=f_idx_i$ a 1-form in $M$, hence $d\eta_i=u_i(x_1)x_idx_1dx_i$. Let is consider

\begin{eqnarray*}
	\gamma_i(s,t)&:=&\pi\circ\phi_t^{\Omega+s(d\eta_i)}(\theta),\text{ for }s\in(-\varepsilon,\varepsilon),\\
	V_i(t)&:=&\frac{\partial}{\partial s}\Big|_{s=0}\gamma_i(s,t),
\end{eqnarray*}

hence $\gamma(t)=\gamma_i(0,t)$ and $V_i(t)$ is a vector field along the magnetic geodesic $\gamma(t)$. Then

$$Z_i(t):=\frac{\partial}{\partial s}\Big|_{s=0} \phi_t^{\Omega+s(d\eta_i)}(\theta)=\left(V_i(t),\frac{D}{dt}V_i(t)\right)\in H(\theta_t)\oplus V(\theta_t).$$

Since that $d\eta_i|_\gamma\equiv0$, then 

$$\frac{D}{ds}\Big|_{s=0}\left(\frac{D}{dt}\dot{\gamma}_i(s,t)=\right)=\frac{D}{ds}\Big|_{s=0}\left(Y_{\gamma_i(s,t)}(\dot{\gamma}_i(s,t))\right),$$

thus we have that $V_i(t)$ satisfied the Jacobi equation (\ref{jacobi}) for $\Omega+d\eta_i$, note that $e_i(0)=e_i(T)$ for every $i=1,\ldots,m$, thus we have that

$$\left\lbrace \begin{array}{l}
\left(\begin{array}{c}V_{i,\bot}(t)\\V'_{i,\bot}(t)\end{array}\right)'=A(t)\left(\begin{array}{c} V_{i,\bot}(t)\\V'_{i,\bot}(t) \end{array} \right) + u_i(t)\left(\begin{array}{c} 0 \\ e_i \end{array} \right),
\text{ for every }t\in[0,T_\theta] \\ \\
V_{i,\bot}(0)=V'_{i,\bot}(0)=0,
\end{array}\right.$$ 

where $V_{i,\bot}(t)=(V_{i,2}(t),\ldots,V_{i,m}(t))$ and $A(t)$ as before. If $S(t)$ is the fundamental matrix of the correspondent homogeneous equation, then

$$\left(\begin{array}{c}V_{i,\bot}\\V'_{i,\bot}\end{array}\right)(T_\theta)=S(T_\theta)\int_0^{T_\theta}u_i(t)S(t)^{-1}\left(\begin{array}{c} 0 \\ e_i\end{array} \right)dt.$$

Fix $t_0\in(0,T_\theta)$ and $0<\lambda<\varepsilon<T_\theta-t_0$ such that $\gamma([t_0-\varepsilon,t_0+\varepsilon])$ does not have self-intersection points. Let $\delta_\lambda:\mathbb{R}\to\mathbb{R}$ be a $C^\infty$-approximation of the Dirac delta at the point $t_0$. Chose $u_i(t)=\delta'_\lambda(t)$ and $\widetilde{u}_i(t)=\delta_\lambda(t)$, we have that, for $(e_i,0),(0,e_i)\in\mathcal{N}(T_\theta)=\mathcal{N}(\theta)$ 

$$d_\theta P(\Omega+d\eta_i)(T_\theta)(e_i,0)=(V_{i,\bot}(T_\theta),V'_{i,\bot}(T_\theta))$$

and

$$d_\theta P(\Omega+d\eta_i)(T_\theta)(0,e_i)=(\widetilde{V}_{i,\bot}(T_\theta),\widetilde{V}'_{i,\bot}(T_\theta)),$$

since 

$$\frac{d}{dt} \left(S(t)^{-1}\left(\begin{array}{c} 0 \\ e_i\end{array} \right) \right)=-S(t)^{-1} \left(\begin{array}{c} e_i \\ 0 \end{array} \right). $$

Thus we have the following result.

\begin{lem}
	Suppose that $\theta_t^\Omega$ is a closed orbit of minimal period $T_\theta>0$ on $T^cM$. Then there is $\eta_2,\ldots\eta_m$ 1-forms in $M$ such that
	
	$$Z_i:=\frac{d}{ds}\Big|_{s=0}\left(\phi_{T_\theta}^{\Omega+sd\eta}(\theta) \right)\text{ for every }i=2,\ldots,m,$$
	
	are a basis of $\mathcal{N}(\theta)$.
\end{lem}

Which implies the following result.

\begin{lem}\label{ev}
	Let $\Omega_0\in\overline{\Omega}^2(M)$ and $\theta_0\in T^cM$ such that $\phi_t^{\Omega_0}(\theta_0)$ is a closed orbit of minimal period $t_0>0$ Then the map
	
	$$\begin{array}{ccccc}
	ev&:&T^cM\times\mathbb{R}\times\overline{\Omega}^2_{[\Omega_0]}(M)&\longrightarrow&T^cM\times T^cM\supset\Delta,\\&&(\theta,t,\Omega)&\longmapsto&(\theta,\phi_t^\Omega(\theta)),\end{array}$$
	
	is transversal to the diagonal $\Delta\subset T^cM\times T^cM$ in the point $(\theta_0,t_0,\Omega_0)$.
	
\end{lem}

The Theorem \ref{pert.} and its Corollary \ref{pert.corl.}, together with the previous lemmas implies the following result.

\begin{lem}\label{arb.ck}
	Let $\Omega_0\in\mathcal{G}^1(c,a)$ and $k\in\mathbb{N}$. Then there exists a $\Omega\in\mathcal{G}_{[\Omega_0]}^k(c,a)$, such that $\Omega$ is arbitrarily $C^r$-close to $\Omega_0$.
\end{lem}

\textbf{Proof of the Proposition \ref{nondeg.}. Density:} 

Let $\Omega\in\overline{\Omega}^2(M)$. Take  $k=k(a,\Omega)\in\mathbb{N}$ such that $(k-1)(K/2)<a\leq k(K/2)$ and $\mathcal{U}$ a $C^r$ open neighborhood of $\Omega$ such that, if $\widehat{\Omega}\in\mathcal{U}$, then $\|\widehat{\Omega}\|_{C^0}<\|\Omega\|_{C^0}+1$, thus $\mathcal{U}\subset\mathcal{G}^l(c,K)$, for every $l\in\mathbb{N}$, in particular

$$\Omega\in\mathcal{U}\subset\mathcal{G}^1(c,K)$$

\begin{itemize}
	\item Consider the map
	
		$$\begin{array}{ccccc}
		ev&:&T^cM\times\mathbb{R}\times\mathcal{U}_{[\Omega]}&\longrightarrow&T^cM\times T^cM\supset\Delta,\\&&(\theta,t,\widehat{\Omega})&\longmapsto&(\theta,\phi_t^{\widehat{\Omega}}(\theta)).\end{array}$$
		
		The Lemma \ref{ev}, implies that, if $ev(\theta_0,t_0,\Omega_0)\in\Delta$, then $ev\pitchfork_{(\theta_0,t_0,\Omega_0)}\Delta$. Hence $ev(\Omega_0)\pitchfork_{T^cM\times[0,3K/2]}\Delta$. So due to Abraham's Theorem of Transversality, we have that the set of every $\Omega_0\in\mathcal{U}_{[\Omega]}$ such that $ev(\Omega_0)\pitchfork_{T^cM\times[0,3K/2]}\Delta$ is dense in $\mathcal{U}_{[\Omega]}$. Then, there is $\widehat{\Omega}_1\in\mathcal{U}_{[\Omega]}$ such that
		
		$$ev(\widehat{\Omega}_1)\pitchfork_{T^cM\times[0,3K/2]}\Delta\text{ and }\|\Omega-\widehat{\Omega}\|_{C^r}<\frac{\varepsilon}{2k}.$$
		
		Lemma \ref{arb.ck}, implies that there is $\Omega_1\in\mathcal{G}^k_{[\Omega]}(c,3K/2)$ with $\|\Omega_1-\widehat{\Omega}_1\|_{C^r}<\frac{\varepsilon}{2k}$. Hence $\|\Omega-\Omega_1\|_{C^r}<\frac{\varepsilon}{k}$.\\
		
	\item We can take $\Omega_1\in\mathcal{U}_{[\Omega]}$ and consider $\mathcal{U}_1=\mathcal{U}\cap\mathcal{G}^k_{[\Omega]}(c,3K/2)$ and
	
		$$\begin{array}{ccccc}
		ev&:&T^cM\times\mathbb{R}\times\mathcal{U}_1&\longrightarrow&T^cM\times T^cM\supset\Delta,\\&&(\theta,t,\widehat{\Omega})&\longmapsto&(\theta,\phi_t^{\widehat{\Omega}}(\theta)).\end{array}$$
		
		Suppose that $ev(\theta_0,t_0,\Omega_0)\in\Delta$. Let $T_0$ be the minimal period of the closed orbit $\phi_t^{\Omega_0}(\theta_0)$. If $T_0\leq 3K/2$ then $ev(\Omega_0)\pitchfork_{(\theta_0,lT_0)}\Delta$, for every $1\leq l\leq k$. Since $\mathcal{U}_1\subset\mathcal{U}_{[\Omega]}$, we have that $K<T_0$ and $t_0<kT_0$. Therefore, $ev(\Omega_0)\pitchfork_{(\theta_0,t_0)}\Delta$. If $T_0\in(3K/2,2K]$ then $t_0=T_0$ and, by Lemma \ref{ev} have that $ev(\Omega_0)\pitchfork_{(\theta_0,t_0,\Omega_0)}\Delta$, hence $ev(\Omega_0)\pitchfork_{T^cM\times [0,2K]}\Delta$. So due to Abraham's Theorem of Transversality, we have that there is $\widehat{\Omega}_2\in\mathcal{U}_1$, such that 
		
		$$ev(\widehat{\Omega}_2)\pitchfork_{T^cM\times[0,2k]}\Delta\text{ and }\|\Omega_1-\widehat{\Omega}_2\|_{C^r}<\frac{\varepsilon}{2k}.$$
		
		Lemma \ref{arb.ck}, implies that there is $\Omega_2\in\mathcal{G}^k_{[\Omega]}(c,2K)$ with $\|\Omega_2-\widehat{\Omega}_2\|_{C^r}<\frac{\varepsilon}{2k}$. Hence $\|\Omega_1-\Omega_2\|_{C^r}<\frac{\varepsilon}{k}$.\\
	
	\item Repeating the same arguments for $2<l\leq k-1$, we obtain 				$\Omega_l\in\mathcal{G}^k_{[\Omega]}(c,l(K/2)+K)$, with 	$\|\Omega_l-\Omega_{l-1}\|_{C^r}<\varepsilon/k$.
	
\end{itemize}

Finally, since $\mathcal{G}^k_{[\Omega]}(c,(k-1)K/2+K)\subset \mathcal{G}^1_{[\Omega]}(c,a)$ and $\|\Omega-\Omega_k\|_{C^r}<\varepsilon$, we have that $\Omega\in\overline{\mathcal{G}^1_{[\Omega]}(c,a)}$.

\end{section}

\begin{section}{Heterocinlic transversal points}

For each $c,a>0$, we define $\mathcal{K}(c,a)$ the set of all $\Omega\in\mathcal{G}^1(c,a)$ such that, for every hyperbolic closed orbits $\theta_t,\vartheta_t\subset T^cM$, of period $<a$, $W^u_a(\theta_t)\pitchfork_{T^cM}W^s_a(\vartheta_t)$. To complete the proof of Theorem \ref{mainthm} is sufficient to prove that, for every $\Omega\in\overline{\Omega}^2(M)$, the set $\mathcal{K}_{[\Omega]}(c,a)$
 is dense in $\mathcal{G}^1_{[\Omega]}(c,a)$. It is enough to prove the existence of a local perturbation for $\Omega$ that preserve the orbits $\theta_t$ and $\vartheta_t$ and such that the perturbation local manifolds $W^u_a(\theta_t)$ and $W^s_a(\vartheta_tate)$ are transversal in a fundamental domain of $W^u_a(\theta_t)$.
 
 \begin{lem}\label{Wu}
 	Let $\sigma\in W^u_a(\theta_t)\subset T^cM$ be such that the restriction $\pi|_{W^u(\theta_t)}$ is a diffeomorphism in a neighbourhood $U\subset W^u(\theta_t)$ of the point $\sigma$. Let $V\subset\overline{V}\subset U$ be sufficiently small neighbourhood of $\sigma$ in $W^u(\theta_t)$. Then there is an exact 2-form $d\eta$, with norm arbitrarily small in the $C^r$ topology $(1\leq r\leq\infty)$, such that 
 	
 	\begin{enumerate}
 		\item $Supp(d\eta)\subset\pi(U)$,
 		\item $\theta_t$ and $\vartheta_t$ are hyperbolic closed orbits of the magnetic flow associated with $\widehat{\Omega}=\Omega+d\eta$,
 		\item $\sigma\in\widehat{W^u_a(\theta_t)}$, where $\widehat{W^u_a(\theta_t)}$ denotes the local stable manifold of $\theta_t$ for the flow $\phi_t^{(\Omega+d\eta)}$,
 		\item the connected component of $\widehat{W^u_a(\theta_t)}\cap V$ that contains the point $\sigma$ and $\widehat{W^s(\theta_t)}$ are transversal.
 	\end{enumerate}
 
\end{lem}
 	
From the general theory of the Lagrangians systems we know that, $W^s(\theta_t),W^s(\theta_t)\subset T^cM$ are Lagrangians submanifolds of $TM$, with the symplectic twist form $\omega(\Omega)$
 	
\begin{lem}{(Twist property of the vertical bundle)}\label{vertical}
	Let $\theta\in M$ and $E\subset T_\theta TM$ be a Lagrangian subspaces for the symplectic twist form. The subset given by 
 		
 		$$\{t\in\mathbb{R}:d_\theta\phi_t^\Omega(E)\cap V(\phi_t^\Omega(\theta))\neq\{0\}\}$$
 		
 	is discrete.
\end{lem}

\textbf{Proof of the density of $\mathcal{K}(c,a)$:} 

Let $\mathcal{D}\subset W^u_a(\theta)$ be a fundamental domain of $W^u_a(\theta)$ and $\sigma\in\mathcal{D}$. By the inverse function theorem we know that $\pi|_{W^u(\theta)}$ is a local diffeomorphism in $\sigma$ if, and only if, $T_\sigma W^u(\theta)\cap V(\sigma)=\{0\}$. As $W^u(\theta)$ is a Lagrangian submanifold we have, from Lemma \ref{vertical}, that $\{t\in\mathbb{R}:d_\sigma\phi_t^\Omega(T_\sigma W^u(\theta))\cap V(\phi_t^\Omega(\sigma))\neq\{0\}\}$, is discrete. Then there exists $t(\sigma)>0$ arbitrarily close to 0, such that $\pi|_{W^u(\theta)}$ is a diffeomorphism in a neighborhood $U_{t(\sigma)}\subset W^u(\theta)$ of the point $\phi_{t(\sigma)}^\Omega(\sigma)$. Since $\Omega\in\mathcal{G}^1(c,a)$, we can assume that $\pi(\phi^\Omega_{-t(\sigma)}(U_{t(\sigma)}))$ does not intersect any closed orbit of period $\leq a$. Let $W_\sigma\subset\mathcal{D}$ be a neighborhood of $\sigma$ such that $\sigma\in W_\sigma\subset\overline{W}_\sigma\subset\phi^\Omega_{-t(\sigma)}(U_{t(\sigma)})$. Then, we can take a finite number of points $\sigma_1,\ldots,\sigma_l$ such that the neighborhood $W_1,\ldots,W_l$ cover the fundamental domain $\mathcal{D}$ and such that the points $\phi^\Omega_{t_i}(\sigma_i)$ and the neighborhoods $\mathcal{V}_i=\phi^\Omega_{t_i}(W_i)\subset U_i$ satisfy the hypothesis in Lemma \ref{Wu}, for each $i=1,\ldots,l$.

Applying Lemma \ref{Wu} to $\phi^\Omega_{t_1}(\sigma_1)\in\mathcal{V}_1\subset\overline{\mathcal{V}}_1\subset U_1$, we obtain an exact 2-form $d\eta_1\in\overline{\Omega}^2(M)$, with $C^r$-norm arbitrarily small, such that $Supp(d\eta)\subset\pi(U_1)$ and the connected component of $\widehat{W^u(\theta_t)}\cap\overline{\mathcal{V}}_1$ that contain $\phi^\Omega_{t_1}(\sigma_1)$ is transversal to $\widehat{W^s(\vartheta_t)}$. Since $\mathcal{G}^1(c,a)$ is open in $\overline{\Omega}^2(M)$, we can assume that $\Omega+d\eta\in\mathcal{G}^1(c,a)$.

The transversality condition on compact subsets is an open condition. Hence, we can successively apply Lemma \ref{Wu} in $\mathcal{V}_i$, to obtain an exact 2-form $d\eta_i\in\overline{\Omega}^2(M)$, with $C^r$-norm small, and such that the invariant manifolds are transversal in $\overline{\mathcal{V}}_1\cup\ldots\cup\overline{\mathcal{V}}_i$, for $1\leq i\leq l$.

Since the number of closed orbits of period $<a$ is finite, repeating the same arguments for each possible pair of hyperbolic orbits of period $<a$, in such a way that the perturbation supports are isolated, we obtain an exact 2-form $d\eta$ in $M$, with $C^r$-norm arbitrarily small, such that $\Omega+d\eta\in\mathcal{K}(c,a)$.\\

Recall that a submanifold $\mathcal{N}$ of a symplectic manifold $(\mathcal{M}^{2n},\omega)$ is \textit{Lagrangian} when $\dim(\mathcal{M})=2\dim(\mathcal{N})$ and $i^*_\mathcal{N}\omega\equiv0$, where $i_\mathcal{N}:\mathcal{N}\to\mathcal{M}$ denotes the inclusion map. The following are easy consequence of the definition and Darboux coordinates. Let $H:\mathcal{M}\to\mathbb{R}$ be a Hamiltonian of class $C^2$.

\begin{itemize}
	\item If $\mathcal{N}\subset H^{-1}(c)$, then $\mathcal{N}$ is Lagrangian if and only if the Hamiltonian vector field $X_H$ is tangent to $\mathcal{N}$.
	\item If $\mathcal{N}\subset H^{-1}(c)$ is Lagrangian and $\theta\in\mathcal{N}$, such that $X_H(\theta)\neq0$. Then there exist a neighborhood $U\subset\mathcal{M}$ of $\theta$ and a coordinate system $(x,y):U\to\mathbb{R}^n\times\mathbb{R}^n$ such that $\omega=\sum_idx_i\wedge dy_i$, $\mathcal{N}\cap U=[y\equiv0]$ and $X_H|_\mathcal{N}=\partial/\partial x_1$.
\end{itemize}

\begin{lem}
	Let $\mathcal{N}$ and $\mathcal{N}_0$ be two Lagrangian submanifolds contained in an energy level $c$ of a Hamiltonian $H:\mathcal{M}\to\mathbb{R}$ on a symplectic manifold $(\mathcal{M}^{2n},\omega)$ , Let $\theta\in\mathcal{N}$ be a non-singular point for the Hamiltonian vector field $X_H$. Let $(t,x,y):U\to[0,1]\times[-\varepsilon,\varepsilon]^{n-1}\times[-\varepsilon,\varepsilon]^n$ be the Darboux coordinates for $\mathcal{N}$ in a neighborhood $U$ of $\theta\in\mathcal{N}$. Then, given $0<\varepsilon_2<\varepsilon_1<\varepsilon$, there exist a sequence of submanifolds $\mathcal{N}_k\subset H^{-1}(c)$ with dimension $n$, such that 
	
	\begin{enumerate}
		\item $\mathcal{N}_k\to\mathcal{N}$ in the $C^\infty$-topology,
		\item $\mathcal{N}\cap A=\mathcal{N}_k\cap A$, where $x=(x_2,\ldots x_n)$, $y=(y_1,\ldots,y_n)$ and
		
		$$A=\{(t,x,y)\in\mathbb{R}^{2n}:\|x\|\geq\varepsilon_1\text{ or }0\leq t\leq1/4\},$$
		
		\item $\mathcal{N}_k$ are invariant in $A\cup B$, where
		
		$$B=\{(t,x,y)\in\mathbb{R}^{2n}:\|x\|\leq\varepsilon\text{ and }1/2\leq t\leq 1\},$$
		
		\item $\mathcal{N}_k\cap C$ are invariant and transversal to $\mathcal{N}_0$, where
		
		$$C=\{(t,x,y)\in\mathbb{R}^{2n}:\|x\|\leq\varepsilon_2\text{ and }1/2\leq t\leq 1\},$$
		
		\item $\int_{\mathcal{N}_k}i^*_k\omega=0$, where $i_k:\mathcal{N}_k\hookrightarrow U$ is the inclusion.

	\end{enumerate}

\end{lem}

\begin{proof}
	Let $\alpha:[-\varepsilon,\varepsilon]\to[0,1]$ and $\beta:[0,1]\to[0,1]$ be smooth functions, such that $\alpha(t)$ is $0$ if $|t|>\varepsilon_1$ and $1$ if $|t|\leq\varepsilon_2$, also that $\int\alpha=0$, on the other hand $\beta(t)$ is 0 if $t\in[0,1/4]$ and 1 if $t\in[1/2,1]$. For $s=(s_2,\ldots,s_n)\in\mathbb{R}^{n-1}$ with $\|s\|$ small, consider $f_s:[0,1]\times[-\varepsilon,\varepsilon]^{n-1}\to\mathbb{R}^n$ defined as 
	
	$$f_s(t,x)=(f^1_s(t,x),s_2\alpha(x_2)\beta(t),\ldots,s_n\alpha(x_n)\beta(t))$$
	
	and $f^1_s$ is defined by:
	
	\begin{eqnarray}\label{f1}
		H(t,x,f_s(t,x))=c.
	\end{eqnarray}
	
	Since the curves $t\mapsto(t,x,0,0)\subset\mathcal{N}$ are solutions of the Hamiltonian system $(\mathcal{M},\omega,H)$, we have that $H(t,x,0,0)=c$ and $\frac{\partial H}{\partial y_1}(t,x,0,0)\neq0$. By the implicit function theorem, for any $s$ with $\|s\|$ sufficiently small, we can solve equation (\ref{f1}) for $(t,x)\mapsto f^1_s(t,x)$ with $f^1_s$ of class $C^\infty$.
	
	We define $\mathcal{N}_s=\{(t,x,f_s(t,x))\in\mathbb{R}^{2n}:(t,x)\in[0,1]\times[-\varepsilon,\varepsilon]^{n-1}\}$. By construction the supports of the maps $f_s$ are fixed and $\lim_{s\to 0}f_s=0$. Therefore, $\mathcal{N}_s\to \mathcal{N}$ in the $C^\infty$-topology when $s\to 0$. Since $f_s(t,x)=0$ for every $(t,x)\in A$, then $\mathcal{N}_s\cap A=\mathcal{N}\cap A$. Moreover, we have that 
	
	\begin{eqnarray*}
		i^*_s\omega&=&i^*_s(dt\wedge dy_1+dx_2\wedge dy_2+\cdots+dx_n\wedge dy_n)\\
		&=&-s_2\alpha(x_2)\beta'(t)dt\wedge dx_2-\cdots-s_n\alpha(x_n)\beta'(t)dt\wedge dx_n.
	\end{eqnarray*}
	
	for every $s$ with $\|s\|$ small, where $i_s:\mathcal{N}_s\hookrightarrow U$ denote the inclusion. Since $\beta'(t)=0$ for every $t\in[1/2,1]$, the submanifolds $\mathcal{N}_s\cap B$ are Lagrangian. Hence $H(\mathcal{N}_s)=c$. Note that 2-form $i^*_s\omega$ has compact support and
	
	$$\int_{\mathcal{N}_s}i^*_s\omega=-\sum_{i=2}^{n}s_i\int_{0}^{1}\beta'(t)\left(\int_{-\varepsilon}^{\varepsilon}\alpha (x_i)dx_i\right)dt=0, $$
	
	for every $s$  with $\|s\|$ small.
	
	Observe that $\mathcal{N}_s\cap C=[\widehat{y}=s]\cap H^{-1}(c)$, where $\widehat{y}=(y_2,\ldots,y_n)$. It is a basic fact about transversality that $\mathcal{N}_s\cap C$ and $\mathcal{N}_0\subset H^{-1}(c)$ are transversal in $H^{-1}(c)$ is and only if $\|s\|$ is small, is a regular value of the map $\rho|_{\mathcal{N}_0}$, where $\rho(y)=\widehat{y}$. Then, by Sard's theorem, we have that there is a sequence $s_n\to 0$ for which $\mathcal{N}_{s_n}$ satisfy the theorem.
	
\end{proof}

\textbf{Acknowledgments}

This work was partially supported by CNPq, CAPES, FAPERJ and PRONEX-DS.
 
\end{section}

\end{document}